\documentclass{amsart}
\newif\ifdraft\draftfalse
\newif\ifcras\crasfalse

\usepackage{mathptmx}
\usepackage{newcent}

\usepackage{amsthm}
\usepackage{amsmath}
\usepackage{amssymb}

\newtheorem{theorem}[equation]{Theorem}

\newtheorem{proposition-definition}[equation]{Proposition-Definition}
\newtheorem{corollary}[equation]{Corollary}
\newtheorem{lemma}[equation]{Lemma}

\theoremstyle{definition}

\newtheorem{remk}[equation]{Remark}

\theoremstyle{remark}

\usepackage{xy}
\ifdraft\usepackage{showlabels}\fi

\newcommand\ds{\displaystyle}
\newcommand\Fields{\operatorname{Fields}}

\newcommand\Sets{\operatorname{Sets}}

\newcommand\Lie{\operatorname{Lie}}
\newcommand\Gl{\operatorname{\mathbf{GL}}}

\newcommand\arr{\ifinner \to\else\longrightarrow\fi}

\newcommand\arrto{\ifinner\mapsto\else\longmapsto\fi}

 \newcommand\calB{\mathcal{B}}

\newcommand\CC{\mathbb{C}} 
 
\newcommand\GG{\mathbf{G}}

\newcommand\QQ{\mathbb{Q}}

 \newcommand\ZZ{\mathbb{Z}}

 \newcommand\rH{\mathrm{H}}

 \newcommand\frp{\mathfrak{p}}

\newcommand\coh{\mathop \mathrm{\kern0pt H}\nolimits}

\newcommand\ed{\mathop \mathrm{ed}\nolimits}
\newcommand\Gal{\mathop \mathrm{Gal}\nolimits}

\newcommand\Aut{\mathop \mathrm{Aut}\nolimits}
\newcommand\trdeg{\mathop \mathrm{trdeg}}

\newcommand\im{\mathop \mathrm{im}\nolimits}

\newcommand\fppf{\mathop \mathrm{fppf}\nolimits}

\newcommand\pearring[2]{\langle #1 , #2 \rangle}

\begin{document}
\title{Essential dimension of abelian varieties over number fields}
\author[Brosnan]{Patrick Brosnan}
\author[Sreekantan]{Ramesh Sreekantan}

\address[Brosnan]{Department of Mathematics\\
The University of British Columbia\\
1984 Mathematics Road\\
Vancouver, B.C., Canada V6T 1Z2}

\address[Sreekantan]{%
School of Mathematics
Tata Institute of Fundamental Research\\
1 Homi Bhabha Road\\
Colaba, Mumbai 400 005\\
India}

\email[Brosnan]{brosnan@math.ubc.ca}
\email[Sreekantan]{ramesh@math.tifr.res.in}


\begin{abstract}
We affirmatively answer a conjecture in the preprint ``Essential
dimension and algebraic stacks,''  proving that the essential
dimension of an abelian variety over a number field is infinite.
\end{abstract}
\maketitle

\ifcras
\newenvironment{proof}{\begin{pf}}{\end{pf}}
\fi

\newcommand\ione{\ifcras{(i)}\else{(1)}\fi}
\newcommand\itwo{\ifcras{(ii)}\else{(2)}\fi}

\long\def\mack{%
  It is a pleasure to thank G.~Pappas, Z.~Reichstein and A.~Vistoli
  for valuable conversations and J.--P. Serre for valuable editorial
  comments.  We are also extremely grateful to N.~Fakhruddin.  After
  seeing a primitive version of this paper proving
  Theorems~\ref{t.main} and~\ref{t.roots} for elliptic curves (using
  Serre's results on Galois representations for ordinary elliptic
  curves and the theory of complex multiplication for elliptic curves
  with CM), he pointed out that Bogomolov's results could be applied
  to prove Theorems~\ref{t.main} and~\ref{t.roots} for abelian
  varieties of any positive dimension.  
} 

\long\def\trivproof{%
\begin{proof}
  The proof is elementary linear algebra with the character lattice,
  $X^*(T)$.

  We can find a basis $e_1,\ldots, e_n$ of $V$ and characters
  $\lambda_1,\ldots,\lambda_n\in X^*(T)$ such that
  $te_i=\lambda_i(t)e_i$ for $t\in T,i\in\{1,\ldots,n\}$.  Since $T$
  contains the homotheties, $\det$ is a non-trivial character of $T$.
  Moreover, since $T\subset\Gl_V$, the $\lambda_i$ generate $X^*(T)$.
  Since $\dim X^*(T)\otimes\QQ\geq 2$, it follows that there exists
  $i$ such that $\lambda_i^{\perp}\not\subset\det^{\perp}$.  Thus we
  can find a cocharacter $\nu$ such that $\pearring{\nu}{\lambda_i}
  =0$ but $\pearring{\nu}{\det}\neq 0$.  Set $S$ equal to the image of
  $\nu$ in $T$ and $v=e_i$.
\end{proof}
}

Let $k$ be a field and let $\Fields_k$ denote the category whose
objects are field extensions $L/k$ and whose morphisms are inclusions
$M\hookrightarrow L$ of fields.  Let $F: \Fields_k\arr\Sets$ be a
covariant functor.  A \emph{field of definition} for an element $a\in
F(L)$ is a subfield $M$ of $L$ over $k$ such that $a\in\im (F(M)\arr
F(L))$.  The \emph{essential dimension} of $a\in F(L)$ is $\ed
a:=\inf\{\trdeg_k M\, |\, M$ is a field of definition for $a\}$.  The
essential dimension of the functor $F$ is $\ed F:=\sup\{\ed a\,|\,
L\in\Fields_k, a\in F(L)\}.$

If $G$ is an algebraic group over $k$, we write $\ed G$ for the
essential dimension of the functor $L\leadsto \rH^1_{\fppf}(L,G)$.  That
is $\ed G$ is the essential dimension of the functor sending a field
$L$ to the set of isomorphism classes of $G$-torsors over $L$.  The
notion of essential dimension of a finite group was introduced by 
J.~Buhler and Z.~Reichstein.  The definition of the essential
dimension of a functor is a generalization given later by
A.~Merkurjev.  In \cite{brv} (which the reader could consult for
further background), a notion of essential dimension for
algebraic stacks was introduced.  In the terminology of that paper,
$\ed G$ is the essential dimension of the stack $\calB G$.  

The purpose of this paper is to generalize the following result.

\begin{theorem}[Corollary 10.4~\cite{brv}]
  Let $E$ be an elliptic curve over a number field $k$.  Assume that
  there is at least one prime $\frp$ of $k$ where $E$ has semistable
  bad reduction.  Then $\ed E=+\infty$.
\end{theorem}

Note that another equivalent way of stating the theorem is to say that
$\ed E=+\infty$ for any elliptic curve $E$ over a number field such that
$j(E)$ is not an algebraic integer.  The
result was proved by showing that Tate curves have infinite essential
dimension.  This method does not apply to elliptic curves
with integral $j$ invariants.  Nonetheless, Conjecture 10.5
of~\cite{brv} guesses that $\ed E=+\infty$ for all elliptic curves
over number fields.  This conjecture is answered by the following.

\begin{theorem}
\label{t.main}
  Let $A$ be a non-trivial abelian variety over a number field $k$.  Then $\ed A=+\infty$. 
\end{theorem}

Note that if $A$ is an abelian variety over $\CC$, then $\ed A=2\dim
A$.  This is the main result of~\cite{BrosnanAB}.

The theorem is an easy consequence of the following result whose
formulation does not involve essential dimension.  To state it, for a
positive integer $m$, let $\mu_m$ denote the group scheme of $m$-th
roots of unity; and, for a rational prime $l$, let $\mu_{l^{\infty}}$
denote the union $\cup_{n\in\ZZ_+} \mu_{l^n}$.

\begin{theorem}
\label{t.roots}
Let $A$ be a non-trivial abelian variety over a number field $k$.  Then there is
an odd prime $\ell$ and an algebraic field extension $L/k$ such that
  \begin{enumerate}
  \item $\QQ_{\ell}/\ZZ_{\ell}\subset A(L)$. 
  \item $1<|\mu_{{\ell}^{\infty}}(L)| < \infty$.
  \end{enumerate}
\end{theorem}

In the first section, we derive Theorem~\ref{t.main} from
Theorem~\ref{t.roots}.  To do this, we use a result of
M.~Florence concerning the essential
dimension of $\ZZ/{\ell}^n$.  In section~\ref{s.later}, we prove
Theorem~\ref{t.roots}.  Here the main results used are those of
Bogomolov and Serre on the action of the absolute Galois
group $\Gal(k)$ on the Tate module $T_{\ell} A$.

\begin{remk}
\rm
  The recent preprint~\cite{KM} of Karpenko and Merkurjev provides
  another way to show that the essential dimension of an abelian
  variety over a number field is infinite.  To be precise, by
  generalizing the results of that paper slightly, one can use them to
  compute the essential dimension of the group scheme $A[n]$ of
  $n$-torsion points of an abelian variety.  In fact, using this idea
  one can show that the essential dimension of an abelian variety over
  a $p$-adic field is also infinite.  However, the present proof 
  of Theorem~\ref{t.main} is shorter than a proof using~\cite{KM} would be
and we hope that Theorem~\ref{t.roots} is
  independently interesting.
\end{remk}

\ifcras\relax\else{\subsection*{Acknowledgments}\mack}\fi

\section{Theorem~\ref{t.roots} implies Theorem~\ref{t.main}}

As mentioned above, we will use the following
result~\cite[Theorem 4.1]{florence} of M.~Florence.

\begin{theorem}\label{t.florence}
  Let $\ell$ be an odd prime and $r$ a positive integer.  Let $L/\QQ$
  be a field such that $|\mu_{\ell^{\infty}} (L)|=\ell^r$.
  Then, for any positive integer $k$, $$
\ed_L \ZZ/{\ell}^k = \max \{ 1, \ell^{k-r}\}.$$
\end{theorem}

\begin{corollary}
\label{c.RootsToMain}
Let $A$ be an abelian variety over a field $L$ of characteristic $0$.
Let $\ell$ be an odd prime and suppose that the statements in the
conclusion of Theorem~\ref{t.roots} are satisfied; i.e:
  \begin{enumerate}
   \item $\QQ_{\ell}/\ZZ_{\ell}\subset A(L)$. 
  \item $1<|\mu_{{\ell}^{\infty}}(L)| < \infty$.
  \end{enumerate}
Then $\ed A=+\infty$.
\end{corollary}
\begin{proof}
  Since $L$ satisfies \itwo, $\ed_L \ZZ/{\ell}^n\to\infty$ as
  $n\to\infty$.  By \ione, there is an injection $(\ZZ/{\ell}^n)_L\to A$.
  Therefore, by~\cite[Theorem 6.19]{bf1}, $\ed A\geq \ed_L \ZZ/{\ell}^n
  -\dim A$ for all $n$.  Letting $n$ tend to $\infty$, we see that
  $\ed A=+\infty$.
\end{proof}

\noindent
\emph{Proof of Theorem~\ref{t.main} assuming Theorem~\ref{t.roots}.}
Let $A$ be a non-trivial abelian variety over a number field $k$.   Using Theorem~\ref{t.roots} and 
Corollary~\ref{c.RootsToMain}, we can find a field extension $L/k$ such that 
$\ed A_L=+\infty$.  This implies that $\ed A=+\infty$ (by~\cite[Proposition 1.5]{bf1}).

\section{Galois representations and the proof of
  Theorem~\ref{t.roots}}
\label{s.later}

Let $A$ be a non-trivial abelian variety over $k$ as in Theorem~\ref{t.roots}.
Before proving Theorem~\ref{t.roots}, we fix some (standard) notation.
We write $\Gal(k):=\Gal(\overline{k}/k)$ for the absolute Galois group of
the number field $k$.  For a rational prime $\ell$, we write $T_{\ell}
A$ for the Tate-module $\ds\lim_{\leftarrow} A[\ell^n]$ of the abelian
variety $A$.  We write $V_{\ell}A$ for
$T_{\ell}A\otimes_{\ZZ_{\ell}}\QQ_{\ell}$.   For an integer $n$, we write $\ZZ/n(1)$ for $\mu_n$, and
for $j\in\ZZ$, we write $\ZZ/n(j)$ for $\mu_n^{\otimes j}$.  We write
$\ds\ZZ_{\ell}(j):=\lim_{\leftarrow} \ZZ/l^m (j)$.

For any prime $\frp$ of $k$ where $A$ has good reduction, write
$T_{\frp}$ for the corresponding Frobenius torus.  (For this notion see~\cite [Definition 3.1
and p. 326] {chi} or~\cite{LetterToRibet}.)  Since $A$ is non-trivial, $T_{\frp}$
contains a rank $1$ torus $D\cong\GG_m$ such that, for every rational
prime $\ell\not\in\frp$, $D(\QQ_{\ell})\subset\Gl(V_{\ell} A)$ is the
set of homotheties (i.e. scalar matrices)~\cite[Proposition 3.2]{chi}.

\begin{lemma}
Let $\frp$ be a prime of $k$ such that the reduction $A/\frp$ of $A$
  at $\frp$ is good but not supersingular.  Then the rank of
  $T_{\frp}$ is strictly greater than $1$.
\end{lemma}
\begin{proof}
This follows directly from~\cite[Proposition 3.3]{chi}.
\end{proof}


The following lemma  was suggested to us by N.~Fakhruddin.  
\begin{lemma}
\label{p.Naf}
  Let $V$ be an $n$-dimensional vector space over a field $F$, and
  let $T$ be an $F$-split torus in $\Gl_{V}$ of rank at least $2$
  containing the homotheties.  
  Then there is a non-zero vector $v\in
  V$ and a rank $1$ subtorus $S$ of $T$ such that
  \begin{enumerate}
  \item $S$ fixes $v$; 
  \item the determinant map $\det: S\to\GG_m$ is surjective.  
  \end{enumerate}
\end{lemma}
\trivproof

\noindent
\emph{Proof of Theorem~\ref{t.roots}.}  Let $A$ be a non-trivial
abelian variety over a number field $k$.  We can find a prime $\frp$
in $k$ such that $A$ has good reduction at $\frp$ but $A/\frp$ is not
supersingular.  (This is well-known if $\dim A=1$: the case where $A$
has CM is standard and otherwise it follows from the exercise on page
IV-13 of ~\cite{serre}.  When $\dim A>1$ it can be proved by adapting
the exercise as Ogus does in Corollary 2.8 of his notes
in~\cite{DMOAS}.)  Thus the Frobenius torus $T_{\frp}$ has rank at
least $2$.  Using Tchebotarev density, it is easy to see that
$T_{\frp}\otimes\QQ_{\ell}$ is a split torus for all rational primes
$\ell$ in a set of positive density.  Thus, we can find an odd
rational prime $\ell$ such that $\ell\not\in\frp$ and
$T_{\frp}\otimes\QQ_{\ell}$ is split.  Now, set $F=k(\zeta_{\ell})$
where $\zeta_{\ell}$ is a primitive $\ell$-th root of unity.  Note
that $T_{\frp}$ is the Frobenius torus for $A_F$ as Frobenius tori are
invariant under finite extension of the ground field.

Now, using Lemma~\ref{p.Naf}, we can can find a rank $1$
subtorus $S\subset T_{\frp}\otimes\QQ_{\ell}$ and a vector $v\in T_{\ell} A_F$ such
that $S$ fixes $v$ and $\det:S\to\GG_m$ is surjective.  Let
$\rho:\Gal(F)\to \Aut(V_{\ell} A_F)$ denote the Galois representation 
on the Tate module and let $H=\{g\in\Gal(F)\,|\, \rho(g)v=v\}$.  
By a theorem of Bogomolov~\cite[Theorem B]{chi} (and the fact that
$S$ fixes $v$),  
it follows that 
$$\Lie(S)\subset \Lie (\rho(H))$$
where $\Lie(S)$ denotes the Lie algebra of $S$ as an algebraic group and 
$\Lie(\rho(H))$ denotes the Lie algebra as an $\ell$-adic group.
Therefore
the intersection of $S(\QQ_{\ell})$ with
$\rho(H)$ contains an open neighborhood of the identity in
$S(\QQ_{\ell})$.   In particular, $\det(H)$ contains a neighborhood of
the identity in $\QQ_{\ell}^*$.  Set $L:=\overline{F}^H$.  Then, from the fact that
$v$ is fixed by $H$, it follows that $\QQ_{\ell}/\ZZ_{\ell}\subset
A(L)$.   On the other hand,  since $\wedge^{2\dim
  A}T_{\ell}A\cong\ZZ_{\ell}(\dim A)$, the fact that 
$\det(H)$ contains a neighborhood 
of the identity in $\QQ_{\ell}^*$ implies that
$\mu_{\ell^{\infty}}(L)$ is finite.  This completes the proof of Theorem~\ref{t.roots}.


\label{}



\ifcras{\section*{Acknowledgments}\mack}\else\fi


\bibliographystyle{plain}


\end{document}